\documentclass[a4paper, 11pt]{amsart}
\usepackage{enumerate, amsfonts, amssymb, amsthm, mathrsfs, mathabx, booktabs}

\usepackage{tikz,tikz-cd}
\usetikzlibrary{shapes,decorations.pathmorphing,calc,arrows,snakes}

\usepackage[colorlinks=true]{hyperref}
\numberwithin{equation}{section}
\usepackage{cleveref}

\newcommand{\Dfn}[1]{\emph{\color{blue}#1}} % to highlight and collect definitions
\newcommand{\FindStat}[1]{\url{www.findstat.org/#1}}
\newcommand{\OEIS}[1]{\url{www.oeis.org/#1}}

\numberwithin{equation}{section}
\theoremstyle{plain}
\newtheorem*{question*}{Question}
\newtheorem{lemma}[equation]{Lemma}
\newtheorem{theorem}[equation]{Theorem}

\newtheorem{corollary}[equation]{Corollary}
\newtheorem{proposition}[equation]{Proposition}

\theoremstyle{definition}

\newtheorem{remark}[equation]{Remark}
\newtheorem{example}[equation]{Example}

\crefname{proposition}{Prop.}{Props.}
\crefname{corollary}{Cor.}{Cors.}
\crefname{theorem}{Thm.}{Thms.}
\crefname{equation}{Eq.}{Eqs.}

\Crefname{proposition}{Proposition}{Propositions}
\Crefname{theorem}{Theorem}{Theorems}
\Crefname{corollary}{Corollary}{Corollaries}
\Crefname{equation}{Equation}{Equations}

\DeclareMathOperator{\Ext}{Ext}%
\DeclareMathOperator{\pdim}{pdim}%
\DeclareMathOperator{\gldim}{gldim}%
\DeclareMathOperator{\height}{height}

\newcommand{\NN}{\mathbb{N}}

\newcommand{\cartan}{\mathbf{C}}
\newcommand{\dist}{\operatorname{dist}}
\newcommand{\depth}{\operatorname{depth}}

\title[On the global dimension of Nakayama algebras]{On the global dimension of Nakayama algebras}

\date{\today}

\author[V.~Klász]{Viktória Klász$^\star$}
\address[Klász]{Mathematical Institute of the University of Bonn, Germany}
\email{klasz@math.uni-bonn.de}
\thanks{$^\star$Supported by the Deutsche Forschungsgemeinschaft (DFG, German Research Foundation) under Germany's Excellence Strategy grant EXC-2047/1-390685813.}

\author[R.~Marczinzik]{René Marczinzik}
\address[Marczinzik]{Mathematical Institute of the University of Bonn, Germany}
\email{marczire@math.uni-bonn.de}

\author[A.~Mellit]{Anton Mellit}
\address[Mellit]{Fakult\"at f\"ur Mathematik, Universität Wien, Austria}
\email{anton.mellit@univie.ac.at}

\author[M.~Rubey]{Martin Rubey}
\address[Rubey]{Fakult\"at f\"ur Mathematik und Geoinformation, TU Wien, Austria}
\email{martin.rubey@tuwien.ac.at}%

\author[C.~Stump]{Christian Stump}
\address[Stump]{Fakult\"at f\"ur Mathematik, Ruhr-Universit\"at Bochum, Germany}
\email{christian.stump@rub.de}

\usepackage[colorinlistoftodos]{todonotes}
\presetkeys%
{todonotes}%
{backgroundcolor=yellow}{}

\begin{document}

\begin{abstract}
  We study the global dimension of Nakayama algebras.  In the case of
  linear Nakayama algebras, which are in canonical bijection to Dyck
  paths, we show that the global dimension has the same distribution
  as the height of Dyck paths.  For cyclic Nakayama algebras an
  explicit classification of finite global dimension is not known.
  However, we show that in certain special cases cyclic Nakayama
  algebras with finite global dimension can again be interpreted as
  Dyck paths.  In particular, we show that there is a natural
  bijection between sincere Nakayama algebras and Dyck paths.  In
  this case, we find that the global dimension is in fact twice the
  bounce count of the corresponding Dyck path.
\end{abstract}

\maketitle
\section{Introduction}
The projective dimension of a module measures its homological
complexity by approximating the module by projective modules, which
are usually well-understood. The \Dfn{global dimension} of a ring $A$
is defined as the supremum of the projective dimensions of all
$A$-modules and gives a global measure for the homological complexity
of the ring. Hilbert's celebrated syzygy theorem is one of the
earliest results in homological algebra.  It states that the
polynomial ring $K[x_1,\dots,x_n]$ in $n$ variables has global
dimension $n$, see for example \cite[Chapter~4.3]{W}.  In one variable, this
theorem says that for every $K[x]$-module $M$ there exists a
projective resolution of the form
$0 \rightarrow P_1 \rightarrow P_0 \rightarrow M \rightarrow 0$ and
this fact can be used to obtain several important structural results
about $K[x]$-modules, which correspond to central results of linear
algebra like the Jordan and canonical rational forms, see
\cite[Chapter~4]{AW}.

The goal of this article is to study the global dimension of Nakayama algebras, also known as serial algebras, see~\cite{AnFul}.
Nakayama algebras form one of the most fundamental classes of finite-dimensional algebras over a field $K$ and are treated in most textbooks on the representation theory of algebras, see for example \cite{AnFul,ARS,ASS,SkoYam}.
A celebrated theorem in modular representation theory states that every representation-finite block of a group algebra is stable equivalent to a Nakayama algebra, see \cite[Chapter~X]{ARS}.
The stable equivalences among Nakayama algebras were classified by Reiten ~\cite{R1,R2}.  As a further example for their relevance, Nakayama algebras were used to answer a question of Gabriel about simply connected algebras~\cite{RoSm}.
 There are still many important open questions about Nakayama algebras, many related to their derived categories, appearing in various contexts such as singularity categories in algebraic geometry, see~\cite{KLM,FOS}.

We work over an algebraically closed field.  In this case every finite dimensional algebra is Morita equivalent to an algebra given by a quiver and relations, see \cite[Proposition~3.7]{ASS}. There are two different types of Nakayama algebras: linear Nakayama algebras and cyclic Nakayama algebras.

Linear Nakayama algebras have many equivalent descriptions, the most elementary is perhaps as an algebra of upper triangular matrices over a field~$K$.
Let $T_n(K)$ be the $K$-algebra of upper triangular $n \times n$-matrices over a field~$K$ and let $J$ be the Jacobson radical of $T_n(K)$, that is, the ideal of all strictly upper triangular matrices.
A \Dfn{(connected) linear Nakayama algebra} is an algebra isomorphic to $T_n(K)/I$ for an ideal~$I \subseteq J^2$.
The $K$-algebra $T_n(K)$ is isomorphic to the path algebra $KQ$ of the following linear quiver $Q$ with $n$ vertices:
\[
\begin{tikzpicture}[scale=0.7]
  \foreach \pos\lab in {0/0, 2/1, 4/, 6/, 8/n-2, 10/n-1}{
    \coordinate (A\pos) at (\pos,0);
    \draw[fill=black] (A\pos) circle (.08);
    \node[anchor=north] at ($(A\pos)-(0,0.2)$) {$\lab$};
  }
  \foreach \sou\tar in {0/2, 2/4, 6/8, 8/10}{
    \draw[->,shorten <=7pt, shorten >=7pt] (A\sou) -- (A\tar);
  }
  \node at (5,0) {$\cdots$};
\end{tikzpicture}
\]
Thus, linear Nakayama algebras may be described as quiver algebras $KQ/I$ for an admissible ideal~$I$.
A \Dfn{(connected) cyclic Nakayama algebra} is a $K$-algebra isomorphic to a quiver algebra $KQ/I,$ of the following cyclic quiver~$Q$ with~$n$ vertices:
\[
\begin{tikzpicture}[scale=0.7]
  \foreach \pos\lab in {0/0, 2/1, 4/, 6/, 8/n-2, 10/n-1}{
    \coordinate (A\pos) at (\pos,0);
    \draw[fill=black] (A\pos) circle (.08);
    \node[anchor=north] at ($(A\pos)-(0,0.2)$) {$\lab$};
  }
  \foreach \sou\tar in {0/2, 2/4, 6/8, 8/10}{
    \draw[->,shorten <=7pt, shorten >=7pt] (A\sou) -- (A\tar);
  }
  \node at (5,0) {$\cdots$};
  \draw[->,shorten <=7pt, shorten >=7pt] (A10.north) to[out=150,in=30] (A0.north);
\end{tikzpicture}
\]
and where~$I$ is an admissible ideal of $KQ$. The simplest example of a cyclic Nakayama algebra is the $K$-algebra $K[x]/(x^n)$ for $n \geq 2$.
We refer to \Cref{sec:prelim} for more details on Nakayama algebras.

In \cite{MRS18}, we found strong connections between homological properties of Nakayama algebras and combinatorial properties of Dyck paths. In particular, linear Nakayama algebras with~$n$ simple modules are in canonical bijection with Dyck paths of semilength~$n-1$, starting at $(0,0)$ and ending at $(2n-2,0)$.
Further homological properties of Nakayama algebras have been related to Dyck path combinatorics in~\cite{RS,CM,KMM}.
In this article, we explore the following question.
\begin{question*}
  Can we describe properties of the global dimension of Nakayama algebras in terms of Dyck paths combinatorics?
\end{question*}

Note that all linear Nakayama algebras have finite global dimension, while cyclic Nakayama algebras may have infinite global dimension.
Moreover, there are only finitely many cyclic Nakayama algebras with a given number of simple modules with finite global dimension, by results of Gustafson~\cite{Gus}.
Recently, several new characterisations when a cyclic Nakayama algebra has finite global dimension have been discovered. We mention three important results:
\begin{enumerate}
\item A linear algebra characterisation: A Nakayama algebra has finite global dimension if and only if it has Cartan determinant~$1$~\cite{BFVZ}.
\item A combinatorial characterisation: A Nakayama algebra has finite global dimension if and only if its resolution quiver has exactly one component and that component has weight~$1$~\cite{S}.
\item A topological characterisation: A Nakayama algebra has finite global dimension if and only if its relation complex is contractible~\cite{HI}.

\end{enumerate}
Those results are all about the finiteness of the global dimension. In this article, we are interested in the combinatorial properties of the global dimension in case it is finite.
The first main result of this article is the following connection between the global dimension of linear Nakayama algebras and the height of Dyck paths:
\begin{theorem}\label{thm:equidistributed}
  The global dimension of connected linear Nakayama algebras is
  equidistributed with the height of Dyck paths.  In symbols,
  \[
    \sum_{A} q^{\gldim(A)} = \sum_{D} q^{\height(D)}
  \]
  where the first sum ranges over all connected linear Nakayama
  algebras with~$n$ simple modules and the second sum ranges over all
  Dyck paths of semilength~$n-1$.
\end{theorem}

It is an open problem to give a precise enumeration of cyclic
Nakayama algebras of finite global dimension and a fixed number of
simple modules.  Our second main result is an indication that the
study of cyclic Nakayama algebras of finite global dimension leads to
beautiful combinatorics.

We focus on \Dfn{sincere} Nakayama algebras~$A$, which are those Nakayama algebras with every non-zero projective module being sincere.  This means that for
every projective $A$-module~$M$, every simple $A$-module is a
composition factor of~$M$, see \cite[Chapter~IX]{ARS} for more on sincere modules.  Note that
a sincere Nakayama algebra must be cyclic.  It turns
out that all sincere Nakayama algebras of finite global dimension
have magnitude~$1$ in the sense of Chuang, King and
Leinster~\cite{CKL}.  We will see in \Cref{gldim_formula} that they
are enumerated by the Catalan numbers via a direct bijection to Dyck
paths.  Our second result is an explicit formula for their global
dimension in terms of the bounce count on Dyck paths.  We refer to
later sections for the explicit definitions.

\begin{theorem}
Let $A$ be a sincere Nakayama algebra with $n$ simple modules.
Then~$A$ has finite global dimension if and only if it has a unique indecomposable projective module of dimension~$n$.
In this case, the global dimension is given by twice the bounce count of the corresponding Dyck path.
\end{theorem}

\section{Preliminaries on Nakayama algebras and Dyck paths}\label{sec:prelim}

\subsection{Algebraic properties of Nakayama algebras}\label{sec:algprops}

In this section, we briefly summarise the most important definitions and properties of Nakayama algebras. We refer for example to \cite{AnFul,ARS,ASS,SkoYam} for textbook introductions to the representation theory of Nakayama algebras and to \cite{MRS18} for the combinatorial aspects of Nakayama algebras and especially the correspondence of linear Nakayama algebras and Dyck paths via their Auslander-Reiten quiver. We assume that all algebras are finite-dimensional $K$-algebras with $K$ an algebraically closed field and $A$-modules are finitely generated right modules unless stated otherwise.

By definition, a \Dfn{Nakayama algebra} $A$ is a finite-dimensional algebra over a field $K$ such that every indecomposable module is uniserial, meaning it has a unique composition series.  The specific field $K$ does not matter for the homological dimensions that we study in this article, so we will often neglect mentioning $K$.

As explained in the introduction, all the homological notions we
study are invariant under Morita equivalence.  Therefore, we will
assume that Nakayama algebras are given as quiver algebras of the
form $KQ/I$.  Furthermore, we can assume that our algebras $A$ are
\Dfn{connected}, meaning that we do not have $A \cong A_1 \times A_2$
for two nontrivial algebras $A_1$ and $A_2$.  A quiver algebras is
connected if its underlying quiver is connected.

A Nakayama algebra is determined by its \Dfn{Kupisch series} $[c_0,\dots,c_{n-1}]$, where $c_i:=\dim(e_i A)$ is the $K$-dimension of the $i$-th indecomposable projective $A$-module and $n$ is the number of simple $A$-modules. We extend the definition of $c_i$ to all $i \in \mathbb{Z}$ by setting $c_i =c_j$ whenever $i \equiv j$ modulo $n$.
Kupisch series of (connected) linear Nakayama algebras are characterised by
\begin{equation}
\label{lnak_eq}
  c_{i+1} + 1 \geq c_i \geq 2 \text{ for $0\leq i < n-1$}, \text{ and } c_{n-1} = 1.
\end{equation}
Kupisch series of (connected) cyclic Nakayama algebras are characterised by
\begin{equation}
\label{cnak_eq}
    c_{i+1}+1\geq c_i \geq 2 \text{ for every }i.
\end{equation}
The \Dfn{Loewy length} of $A$ is the maximum of the integers $c_i$.
The indecomposable $A$-modules are all of the form $e_i A/e_i J^k$ for some $i\in\{0,\dots,n-1\}$ and $k \in \{1,\dots,c_i \}$.
The modules $e_i A = e_i A / e_i J^{c_i}$ are the indecomposable projective $A$-modules and the modules $S_i = e_i A/ e_i J^1$ are the simple $A$-modules.

Setting $d_i := \min \{ \ k \ | \ k \geq c_{i-k} \ \}$, the sequence $[d_0,\dots,d_{n-1}]$ is the \Dfn{coKupisch series} of $A$ and the indecomposable injective $A$-modules are given by $D(Ae_i)= e_{i+1-d_i} A/e_{i+1-d_i} J^{d_i}$. Note that the coKupisch series is exactly the Kupisch series of the opposite algebra of the Nakayama algebra $A$ and $d_i$ is equal to the vector space dimension of the indecomposable injective module $D(Ae_i)$, see \cite[Theorem 2.2]{Ful}.

The \Dfn{projective dimension} of a module $M$ with minimal projective resolution
\[
\cdots\rightarrow P_i \rightarrow P_{i-1} \rightarrow \cdots\rightarrow P_1 \rightarrow P_0 \rightarrow M \rightarrow 0
\]
is defined as $\pdim M:= \sup \{\ i \geq 0 \mid P_i \neq 0 \}$. The \Dfn{global dimension} of an algebra $A$ is defined as $\gldim A := \sup \{ \ \pdim M \mid M \in $ mod$A\}$, where mod$A$ denotes the module category of $A$ consisting of finitely generated $A$-modules.
We define the \Dfn{$n$-th syzygy} of a module $M=e_i A/e_i J^k$ over a Nakayama algebra inductively as $\Omega^0(M)=M$, $\Omega^1(e_i A/ e_i J^k)=e_i J^k= e_{i+k} A/ e_{i+k} J^{c_i -k}$ and for $\ell \geq 2$ as $\Omega^\ell(M)=\Omega^1(\Omega^{\ell-1}(M))$.
The projective dimension of $M$ is equal to the smallest non-negative integer $\ell$ such that $\Omega^\ell(M)$ is projective, or infinite if no such $\ell$ exists. In the finite case, $\ell$ is the smallest non-negative integer such that  $\Omega^{\ell+1}(M)=0.$
This gives a purely combinatorial way to calculate projective dimensions of modules over Nakayama algebras.

Note that linear Nakayama algebras are characterised among Nakayama algebras as those having a simple projective module, which corresponds to the sink vertex of the quiver, i.e., the vertex with no outgoing arrows.
The following is a fundamental result due to Auslander on the global dimension of finite-dimensional algebras, see \cite[Proposition~5.1]{ARS}:
\begin{theorem} \label{Auslanderthm}
Let $A$ be a finite-dimensional algebra.
Then
\[
  \gldim(A) = \sup \{ \ \pdim(S_i) \mid S_i \text{ simple}\ \}\,.
\]
\end{theorem}

The following result by Gustafson first establishes that there are only finitely many Nakayama algebras with finite global dimension and a fixed number of simple modules~\cite{Gus}:
\begin{theorem} \label{gustafsonbound}
Let $A$ be a Nakayama algebra with $n$ simple modules.
If $A$ has finite global dimension then $A$ has Loewy length at most $2n-1$.
\end{theorem}

The \Dfn{Cartan matrix} of a finite-dimensional quiver $K$-algebra $A$ with $n$ simple modules is the $n\times n$-matrix $\cartan_A=(c_{ij})$ with entries $c_{ij}=\dim_K e_i A e_j$, the vector space dimension of the space $e_i A e_j$ given by the number of non-zero paths in $A$ starting at $i$ and ending at $j$. The \Dfn{Cartan determinant} of $A$ is $\det(\cartan_A)$. It is well-known that in the case of finite global dimension, this is equal to $\pm 1$, see \cite[Proposition 21]{E}.
The Cartan determinant conjecture, a longstanding open problem, asserts that it is always positive, see for example the section on homological conjectures in \cite{ARS}.

The resolution quiver of a Nakayama algebra was introduced by Ringel to describe homological properties of Nakayama algebras in a graph theoretic way~\cite{Rin1}.

Let $A$ be a Nakayama algebra with Kupisch series $[c_0,c_1,\dots,c_{n-1}]$.
Its \Dfn{resolution quiver} has vertex set~$\mathbb{Z}/n \mathbb{Z}$ and an arrow from $i$ to $j$ if $j \equiv i+c_i$ modulo $n$.
Note that each component contains a unique cycle.
Thus, the number of components of the resolution quiver is the number of its cycles.
The \Dfn{weight} of a cycle in the resolution quiver is $\frac{1}{n}\sum\limits_{\ell=1}^{s}{c_{i_\ell}}$, where $i_1,\dots, i_s$ are the vertices in the cycle.

\medskip

We recall the following properties for later reference.

\begin{theorem} \label{Dshentheorems}
Let $A$ be a cyclic Nakayama algebra.
\begin{enumerate}
\item All cycles in the resolution quiver of $A$ are of the same size and weight.
\item $A$ has finite global dimension if and only if the resolution quiver of $A$ is connected and its unique cycle has weight $1$.
\item If $A$ has finite global dimension, then the number of vertices lying on the unique cycle in the resolution quiver of $A$ is the number of simple modules of even projective dimension of $A$.
\item The Cartan matrix of $A$ is invertible if and only if the resolution quiver of $A$ is connected. In this case, the Cartan determinant of $A$ equals the weight of any cycle of $A$.
\end{enumerate}
\end{theorem}
\begin{proof}
\begin{enumerate}
\item See \cite[Proposition~1.1]{S2}
\item See \cite[Proposition~1.1]{S}.
\item See \cite[Lemma~3.1(2)]{S}.
\item The rank of the Cartan matrix is $n+1-c$, where $c$ is the number of cycles of the resolution quiver. Since every component of the resolution quiver has a unique cycle, the number of cycles is the number of components. Thus, the Cartan matrix has full rank if and only if $c=1$, which means that the resolution quiver is connected. The fact that the Cartan determinant equals the weight in case it is non-zero is explained in \cite[Remark 4.2.]{S}.
\qedhere
\end{enumerate}
\end{proof}

\begin{proposition} \label{benpropo}
Let $A$ be a finite-dimensional algebra, $M$ an $A$-module with minimal projective resolution
\[
\dots\rightarrow P_k \rightarrow\dots\rightarrow P_1 \rightarrow P_0 \rightarrow M \rightarrow 0,
\]
and $S$ a simple $A$-module. Then $\dim\Ext_A^k(M,S)$ equals the number of indecomposable direct summands of $P_k$ that are isomorphic to the projective cover of $S$.
\end{proposition}
\begin{proof}
See for example \cite[Corollary 2.5.4]{Ben}.
\end{proof}

In \cite{CKL}, the \Dfn{magnitude} of a finite-dimensional algebra
$A$ with finite global dimension is defined as the sum of all entries
of the inverse of the Cartan matrix of $A$.  We denote the magnitude
of $A$ by $m_A$.  The notion of magnitude originates in the study of
enriched categories.  The main result of \cite{CKL} relates it to the
Euler form of $A$.

\begin{example}
  \label{example-A-344321}
  Let $A$ be the linear Nakayama algebra with Kupisch series [3,4,4,3,2,1].
  Then $A=KQ/I$ with quiver
  \[
    Q = \begin{tikzcd}
      0 & 1 & 2 & 3 & 4 & 5
      \arrow["{\alpha_1}", from=1-1, to=1-2]
      \arrow["{\alpha_2}", from=1-2, to=1-3]
      \arrow["{\alpha_3}", from=1-3, to=1-4]
      \arrow["{\alpha_4}", from=1-4, to=1-5]
      \arrow["{\alpha_5}", from=1-5, to=1-6]
    \end{tikzcd},
  \]
  and relations
  $I = \langle \alpha_1 \alpha_2 \alpha_3, \alpha_2 \alpha_3 \alpha_4
  \alpha_5 \rangle$.  Its resolution quiver is
  \[
    \begin{tikzcd}
        &   & 4\\
      1 & 5 & 0 & 3\\
        &   & 2
      \arrow[from=2-1, to=2-2]
      \arrow[from=2-2, to=2-3]
      \arrow[from=2-3, to=2-4]
      \arrow[from=2-4, to=2-3, in=30, out=150]
      \arrow[from=1-3, to=2-3]
      \arrow[from=3-3, to=2-3]
    \end{tikzcd}
  \]
  and its weight is $1$.  $A$ has global dimension~$3$, and its
  Cartan matrix is
  \[
    \cartan_A=\begin{pmatrix}
      1 & 1 & 1 & 0 & 0 &0 \\
      0 & 1 & 1 & 1 & 1 &0 \\
      0 & 0 & 1 & 1 & 1 &1 \\
      0 & 0 & 0 & 1 & 1 &1 \\
      0 & 0 & 0 & 0 & 1 &1 \\
      0 & 0 & 0 & 0 & 0 &1
    \end{pmatrix}
  \]
  and the magnitude of $A$ is two.
\end{example}
\begin{example}
  Let $A$ be the cyclic Nakayama algebra with Kupisch series
  $[3, 3, 3, 4]$. Then $A=KQ/I$ with quiver
  % https://q.uiver.app/#q=WzAsNCxbMCwwLCIwIl0sWzEsMCwiMSJdLFsxLDEsIjIiXSxbMCwxLCIzIl0sWzAsMSwiXFxhbHBoYV8xIl0sWzEsMiwiXFxhbHBoYV8yIl0sWzIsMywiXFxhbHBoYV8zIl0sWzMsMCwiXFxhbHBoYV80Il1d
  \[
    Q=\begin{tikzcd}
      0 & 1 \\
      3 & 2
      \arrow["{\alpha_1}", from=1-1, to=1-2]
      \arrow["{\alpha_2}", from=1-2, to=2-2]
      \arrow["{\alpha_4}", from=2-1, to=1-1]
      \arrow["{\alpha_3}", from=2-2, to=2-1]
    \end{tikzcd}\] and relations
  $I=\langle \alpha_1 \alpha_2 \alpha_3, \alpha_2 \alpha_3 \alpha_4,
  \alpha_3 \alpha_4 \alpha_1 \rangle$.  Its resolution quiver is
  % https://q.uiver.app/#q=WzAsNCxbMiwwLCIwIl0sWzMsMCwiMyJdLFsxLDAsIjEiXSxbMCwwLCIyIl0sWzAsMV0sWzEsMV0sWzIsMF0sWzMsMl1d
  \[
    \begin{tikzcd}
      2 & 1 & 0 & 3
      \arrow[from=1-1, to=1-2]
      \arrow[from=1-2, to=1-3]
      \arrow[from=1-3, to=1-4]
      \arrow[from=1-4, to=1-4, loop, in=55, out=125, distance=10mm]
    \end{tikzcd}
  \]
  and the weight is $1$. Thus $A$ has finite global dimension, which
  is~$5$.  Its Cartan matrix is
  \[
    \cartan_A=\begin{pmatrix}
      1 & 1 & 1 & 0\\
      0 & 1 & 1 & 1\\
      1 & 0 & 1 & 1\\
      1 & 1 & 1 & 1
    \end{pmatrix}
  \]
  and the magnitude of $A$ is one.
\end{example}

\subsection{Nakayama algebras and Dyck paths}

We next recall some definitions concerning Dyck paths and some connections with Nakayama algebras.

A \Dfn{Dyck path} of semilength~$n$ is a lattice path starting at
$(0,0)$ and ending at $(2n,0)$, consisting of diagonal \Dfn{up steps}
$(1,1)$ and \Dfn{down steps} $(1,-1)$, with the additional property
that it never goes below the $x$-axis.

The \Dfn{area sequence}\footnote{The area sequence as defined here is
  a variant of the traditional version, which is obtained by dropping
  the final entry $c_n=1$, subtracting $2$ from all other entries and
  reversing the sequence.} $[c_0, c_1,\ldots,c_n]$ of a Dyck path is
obtained by letting $c_k$ be the biggest integer such that
$(2k+c_k-1,c_k-1)$ lies on the Dyck path, for $0\le k\le n$. In other
words, $c_k$ is the number of lattice points lying on the diagonal
starting at $(2k,0)$ in direction $(1,1)$ between the $x$-axis and
the Dyck path.  We refer to \Cref{exampleDyckBounce,example2} for
some visualisations of Dyck paths and their area sequences.

Sending a Dyck path to its area sequence is a bijection between Dyck
paths of semilength $n$ and sequences $[c_0, c_1,\ldots,c_n]$
satisfying
\[
    c_{i+1}+1\ge c_i \ge 2\text{ for $0\le i\le n-1$}, \text{ and } c_n=1.
\]
Comparing these properties to~\cref{lnak_eq}, we obtain a simple
bijection between connected linear Nakayama algebras with~$n$ simple
modules and Dyck paths of semilength~$n-1$.  Explicitly, the
indecomposable $A$-module $e_i A/e_i J^k$ corresponds to the point
$(i, k-1)$ on or below the Dyck path.
In a more representation theoretic language, this bijection sends a linear Nakayama algebra algebra $A$ to the Dyck path given by the top boundary of the Auslander-Reiten quiver of $A$.

The \Dfn{height} of a Dyck path is the maximum of the $c_i$'s in the
area sequence, minus one.  Note that, under the bijection above, this equals the
Loewy length of the linear Nakayama algebra, minus one.  In
\Cref{sec:equidistributed} we show that the distribution of the
global dimension of connected linear Nakayama algebras coincides with
the distribution of the height of Dyck paths.

In \Cref{sec:sincere} we relate the global dimension of sincere
Nakayama algebras to the \Dfn{bounce count} $b_D$ of a Dyck path $D$,
which we define next.  Let $D$ be a Dyck path with area sequence
$[c_0, c_1,\ldots,c_n]$.  The \Dfn{bounce path} of $D$ is a Dyck path
associated to $D$ defined by starting at $(0,0)$ with $b_1=c_0-1$ up
steps followed by $c_0-1$ down steps.  For $t>0$ we inductively
define it as starting at $(2b_t,0)$, taking $c_{b_t}-1$ up steps
followed by $c_{b_t}-1$ down steps, to arrive at $(2b_{t+1},0)$.
Since the bounce path remains below the Dyck path, we eventually have
$(2b_d,0) = (2n,0)$.  The \Dfn{bounce count}~$b_D$ is then $d$, the
number of iteration steps.

\begin{example}
  \label{exampleDyckBounce}
  The following picture illustrates the Dyck path (black) of
  semilength~$11$ with area sequence $[3,5,5,5,4,3,4,3,3,3,2,1]$ and
  its bounce path (blue, dotted) with its bounces $b_1=2$, $b_2=6$,
  $b_3=9$, and $b_4 = 11$. The height of this Dyck path is $4$ and
  its bounce count is $4$.
  
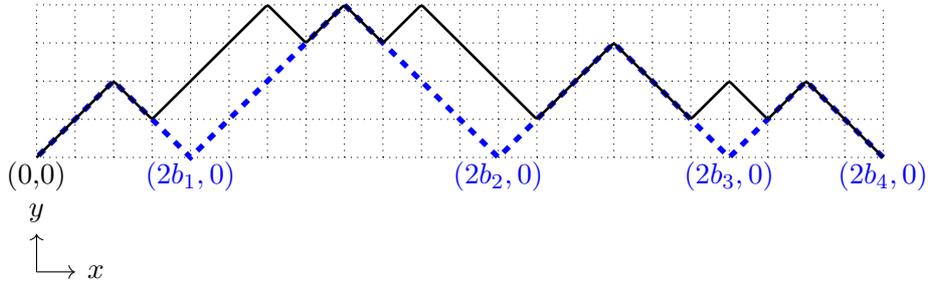
\[
    \begin{tikzpicture}[scale=\textwidth/25cm]
      \node at (22,-0.5) {\textcolor{blue}{($2b_4,0$)}};
      \node at (18,-0.5) {\textcolor{blue}{($2b_3,0$)}};
      \node at (12,-0.5) {\textcolor{blue}{($2b_2,0$)}};
      \node at (4,-0.5) {\textcolor{blue}{($2b_1,0$)}};
      \node at (0,-0.5) {(0,0)};
      \draw[->] (0, -3) -- node[right, xshift=8pt] {$x$} (1, -3);
      \draw[->] (0, -3) -- node[above, yshift=8pt] {$y$} (0, -2);
      \draw[dotted] (0, 0) grid (22, 4);
      \draw[color=blue, thick, dashed, line width=2] (22, 0) -- (20,2) -- (18,0) -- (15,3) -- (12,0) -- (8,4) -- (4,0) -- (2,2) --(0,0);
      \draw[rounded corners=1, color=black, line width=1] (22, 0) -- (20, 2) -- (19, 1) -- (18, 2) -- (17, 1) -- (15, 3) -- (13, 1) -- (10, 4) -- (9, 3) -- (8, 4) -- (7, 3) -- (6, 4) -- (3, 1) -- (2, 2) -- (0, 0);
    \end{tikzpicture}
  \]
\end{example}

\begin{example}
  \label{example2}
  Let $D$ be the Dyck path with area sequence $[3,4,4,3,2,1]$ of
  semilength~$5$.  The height of this Dyck path is~$3$.
  
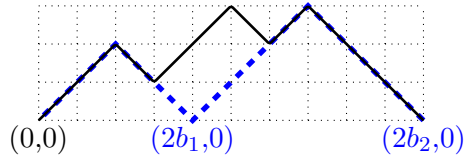
\[
    \begin{tikzpicture}[scale=\textwidth/25cm]
      \node at (10,-0.5) {\textcolor{blue}{($2b_2$,0)}};
      \node at (0,-0.5) {(0,0)};
      \node at (4,-0.5) {\textcolor{blue}{($2b_1$,0)}};
      \draw[dotted] (0, 0) grid (10, 3);
      \draw[color=blue, thick, dashed, line width=2] (10, 0) -- (7,3) -- (4,0) -- (2,2) -- (0,0);
      \draw[rounded corners=1, color=black, line width=1] (10, 0) -- (7, 3) -- (6, 2) -- (5, 3) -- (3, 1) -- (2, 2) -- (0, 0);
    \end{tikzpicture}
  \]
  Its associated bounce path is visualised by the blue dotted line,
  the two bounce points are $b_1= 2$ and $b_2 = 5$.  Thus, the bounce
  count is $2$.
\end{example}

\section{The global dimension of Nakayama algebras with a linear quiver}
\label{sec:equidistributed}
In this section we prove \Cref{thm:equidistributed}.  To do so, we
consider products of connected linear Nakayama algebras
$A = A_1 \times A_2 \times \dots \times A_k$ with~$n$ simple modules
in total.  Let $\mathcal{A}_n$ denote the set of such products.  We
emphasize that the order of the factors matters,
$A_1 \times A_2 \neq A_2 \times A_1$.

The \Dfn{Kupisch series} $[c_0,c_1,\ldots,c_{n-1}]$ of~$A$ is the
concatenation of the Kupisch series of $A_1,\dots,A_k$.  In
particular, the number of entries equal to~$1$ in the Kupisch series
of~$A$ equals the number of factors~$k$, and the global dimension is
\begin{equation*}
  \gldim(A) = \max\big\{\gldim(A_1),\dots,\gldim(A_k)\big\}.
\end{equation*}

We first associate a directed graph with $A$, similar to the
resolution quiver of a Nakayama algebra mentioned in
\Cref{sec:algprops}.  Let $\tau(A)$ be the (directed) graph with
vertices $\{0,\dots, n\}$ and edges $i+c_i \rightarrow i$ for
$0\leq i < n$.
\begin{figure}
  \centering
  \scalebox{0.7}
{ \newcommand{\nodea}{\node[draw,circle] (a) {$6$}
;}\newcommand{\nodeb}{\node[draw,circle] (b) {$2$}
;}\newcommand{\nodec}{\node[draw,circle] (c) {$3$}
;}\newcommand{\noded}{\node[draw,circle] (d) {$0$}
;}\newcommand{\nodee}{\node[draw,circle] (e) {$4$}
;}\newcommand{\nodef}{\node[draw,circle] (f) {$5$}
;}\newcommand{\nodeg}{\node[draw,circle] (g) {$1$}
;}\begin{tikzpicture}[auto]
\matrix[column sep=.3cm, row sep=.3cm,ampersand replacement=\&]{
         \&         \& \nodea  \&         \&         \\
 \nodeb  \& \nodec  \&         \& \nodee  \& \nodef  \\
         \& \noded  \&         \&         \& \nodeg  \\
};
\path[ultra thick, red] (c) edge (d)
	(f) edge (g)
	(a) edge (b) edge (c) edge (e) edge (f);
\end{tikzpicture}}
\caption{The tree corresponding to the Kupisch series $[3,4,4,3,2,1]$.}
\label{fig:tree0}
\end{figure}
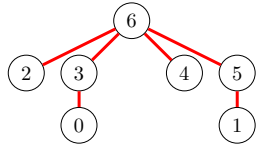

\begin{lemma}
    Let $A \in \mathcal A_n$.
    Then $\tau(A)$ is a tree rooted at the vertex~$n$.
\end{lemma}
\begin{proof}
  The graph is acyclic because vertex labels decrease along directed
  edges.  Since the graph has one more vertex than it has edges, it must be a tree.
\end{proof}

We draw the directed edges of~$\tau(A)$ downwards as indicated in \Cref{fig:tree0}.
Moreover, we order the children of each vertex increasingly from left to right, turning $\tau(A)$ into an \emph{ordered rooted tree}.

For $A = A_1\times A_2$, the tree~$\tau(A)$ is obtained from the trees~$\tau(A_1),\tau(A_2)$ as follows.

\begin{proposition}\label{prop:concatenate}
  Let $A_1$ and $A_2$ be linear Nakayama algebras with $n_1$ and $n_2$ simple modules, respectively.
  Then the graph $\tau(A_1 \times A_2)$ is obtained from the graphs $\tau(A_1)$ and $\tau(A_2)$ by increasing all vertex labels of $\tau(A_2)$ by $n_1$ and then identifying the vertex with new label~$n_1$ (originally vertex~$0$ of $\tau(A_2)$) with the vertex $n_1$ of $\tau(A_1)$.
\end{proposition}

\begin{proof}
  This follows immediately from the definition.
\end{proof}

The \Dfn{depth} of a vertex~$i$ in a rooted tree~$T$ is the number of edges on the path from~$i$ to the root and denoted by~$\depth_T(i)$.
The depth of a rooted tree $T$, denoted by $\depth(T)$, is the maximal depth of a vertex of $T$.
The \Dfn{distance} between two vertices~$i,j$ in~$T$, denoted by $\dist_T(i,j)$ is the number of edges on the unique (undirected) path between~$i$ and~$j$.
In particular, $\depth_T(i) = \dist_T(i,n)$.

\begin{lemma}
\label{thm:projdim}
  Let~$A \in \mathcal A_n$ and let $T = \tau(A)$.
  The projective dimension of the simple module~$S_i$ of~$A$ for $0\le i<n$ is
  \[
    \pdim(S_i) = \dist_T(i,i+1) - 1\,.
  \]
  In particular,
  \[
    \gldim(A)=\max_{0\leq i < n} \dist_T(i,i+1) - 1\,.
  \]
\end{lemma}
\begin{proof}
  By \Cref{prop:concatenate} it is enough to prove that $\pdim(S_i) = \dist_T(i,i+1) - 1$ for each connected component of~$A$.
  We thus assume that~$A$ is connected.

  We show that for $0\leq i < n$, the sequence of syzygies of the
  simple module $S_i$ of projective dimension $\ell$,
  \[
    S_i = \Omega^0(S_i), \Omega^1(S_i), \dots, \Omega^\ell(S_i), \Omega^{\ell+1}(S_i) = 0,
  \]
  corresponds to the unique shortest path from $i$ to $i+1$ in~$T$. Every linear Nakayama algebra has finite global dimension, implying $\ell<\infty$.
  Let $b_{i,k} = e_i A/ e_i J^k$.  Fixing $i$, we let $0\leq s_d < n$ be such that $\Omega^d(S_i) = b_{s_d, k}$ for some $k$, for $0\leq d\leq \ell$.  In particular, $b_{i, 1} = S_i$, so $s_0=i$.

  Recall from \Cref{sec:algprops} that $\Omega(b_{i, k}) = b_{i+k, c_i-k}$.  In particular, $\Omega(S_i) = \Omega(b_{i, 1}) = b_{i+1, c_i-1}$, so $s_1=i+1$.  Moreover,
  $\Omega^2(b_{i, k}) = \Omega(b_{i+k, c_i-k}) = b_{i+c_i, c_{i+k}-c_i+k}$, which implies that $s_{d+2} = s_d + c_{s_d}$.  Since $T$ has edges from $i$ to $i+c_i$ for $0\leq i < n$, it follows that the sequences
  \[
    s_0, s_2,\dots,\quad\text{and}\quad s_1, s_3,\dots
  \]
  are the vertices of two disjoint
  paths in $T$ from $i$ and $i+1$, respectively, towards the root.

  Let $x=s_{\ell-1}$ and $y=s_\ell$.  It remains to show that
  $x + c_x = y + c_y$ to conclude that the paths indeed meet at this
  vertex.  Since $\Omega^{\ell} = \Omega\circ\Omega^{\ell-1}$ we have
  that $\Omega^{\ell-1}(S_i) = b_{x, y-x}$, and that
  $\Omega^{\ell}(S_i) = b_{y, c_x - y + x}$.  Since
  $\Omega^{\ell}(S_i)$ is projective,
  $\Omega^{\ell}(S_i) = b_{y, c_y}$.  Therefore $c_y = c_x - y + x$, and the unique shortest path from $i$ to $i+1$ indeed consists of $\ell + 1$ edges.

  Finally, the statement about the global dimension follows with~\Cref{Auslanderthm}.
\end{proof}

\begin{example}
\Cref{fig:tree0} displays the tree corresponding to the Kupisch series $[3,4,4,3,2,1]$ considered in \Cref{example-A-344321}.  Setting $b_{i,k} = e_iA/e_i J^k$, the sequence of syzygies of the simple module $S_0 = b_{0,1}$ is
\[
\Omega^1(S_0)=b_{1,2},\quad \Omega^2(S_0)=b_{3,2},\quad \Omega^3(S_0)=b_{5,1},\quad \Omega^4(S_0)=b_{6,0} = 0.
\]
In particular, the projective dimension of $S_0$ is three.  Indeed, the path from vertex~$0$ to vertex~$1$ passes through vertices~$3$,~$6$ and~$5$, and has four edges.
\end{example}

We call a rooted tree with vertices~$\{0,\dots,n\}$ \Dfn{naturally labeled}, if the labels increase towards the root and the labels of the children of~$i$ are all smaller than the labels of the children of~$i+1$.  The trees in \Cref{fig:tree0,fig:tree1} are all naturally labeled.
\begin{lemma}
    Let $A\in\mathcal A_n$.  Then $\tau(A)$ is naturally labeled.
\end{lemma}
\begin{proof}
  The Kupisch series $[c_0,\dots,c_{n-1}]$ of~$A$ satisfies
  $c_i+i \leq c_{i+1} + i+1$, implying $c_i\leq c_{i+k} + k$ for
  $k\geq 0$.  It remains to show that all children of~$i$ have labels
  smaller than the labels of the children of~$i+1$.  Let~$k$ be the
  largest child of $i = c_k + k$ and let $\ell$ be the smallest child
  of $i + 1 = c_k+k+1 = c_\ell+\ell$.  Suppose that $\ell < k$.
  Then,
  \[
  c_\ell + \ell = c_k+k+1 = c_{\ell + (k-\ell)} + (k-\ell) + \ell + 1 \geq c_\ell + \ell + 1,
  \]
  which is a contradiction.
\end{proof}

Let $\mathcal L_n$ be the set of naturally labeled trees with vertices $\{0,\dots,n\}$.
\begin{lemma}
    The map $\tau:\mathcal A_n\to\mathcal L_n$ is a bijection.
\end{lemma}
\begin{proof}
    Let $T$ be a naturally labeled tree.  For each vertex $i\in\{0,\dots,n-1\}$ with parent $j$, define $c_i=j-i$.  It suffices to show that $[c_0,\dots,c_{n-1}]$ is the Kupisch series of a (not necessarily connected) linear Nakayama algebra, that is, it satisfies $c_i\leq c_{i+1}+1$ for $0\leq i < n-1$.  To do so, let $j$ be the parent of $i$ and let $k$ be the parent of $i+1$.  Then, because $T$ is naturally labeled, we have $j\leq k$.  It follows that $c_i = j - i\leq k - i = c_{i+1} + 1$.
\end{proof}

\begin{lemma}\label{lem:interval}
  The labels of the vertices of a naturally labeled tree having the same distance to the root form an interval of $\{0,\dots,n\}$.
\end{lemma}
\begin{proof}
    We use induction on the distance to the root.  Assume that the labels of vertices with distance $d$ to the root are consecutive.  Let $i$ and $k$ be vertices with distance $d+1$ to the root and suppose that there is a vertex $j$ with $i < j < k$.  Let $a$, $b$ and $c$ be the parents of $i$, $j$ and $k$, respectively.  Because the tree is naturally labeled, we must have $a\leq b\leq c$.  Therefore, vertex $b$ must have distance $d$ to the root, which in turn implies that $j$ has distance $d+1$ to the root.
\end{proof}

% sage: seq = [5,4,5,4,3,3,3,3,2,3,2,1]; D = from_kupisch(seq)
% sage: ascii_art(dw2labelled_tree(shift(D)))
% sage: ascii_art(decompose_tree(T))
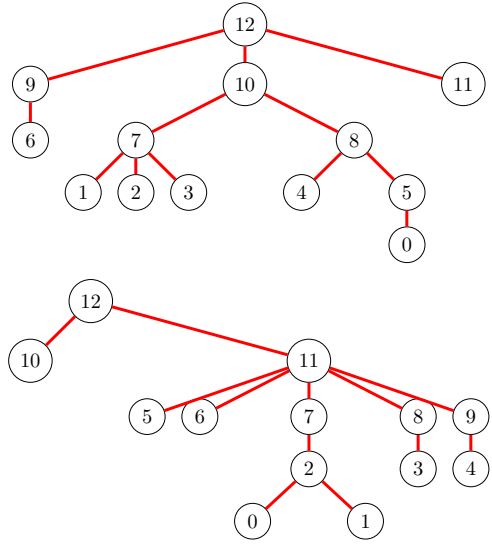
\begin{figure}
  \centering
  \scalebox{0.7}
{ \newcommand{\nodea}{\node[draw,circle] (a) {$12$}
;}\newcommand{\nodeb}{\node[draw,circle] (b) {$9$}
;}\newcommand{\nodec}{\node[draw,circle] (c) {$6$}
;}\newcommand{\noded}{\node[draw,circle] (d) {$10$}
;}\newcommand{\nodee}{\node[draw,circle] (e) {$7$}
;}\newcommand{\nodef}{\node[draw,circle] (f) {$1$}
;}\newcommand{\nodeg}{\node[draw,circle] (g) {$2$}
;}\newcommand{\nodeh}{\node[draw,circle] (h) {$3$}
;}\newcommand{\nodei}{\node[draw,circle] (i) {$8$}
;}\newcommand{\nodej}{\node[draw,circle] (j) {$4$}
;}\newcommand{\nodeba}{\node[draw,circle] (ba) {$5$}
;}\newcommand{\nodebb}{\node[draw,circle] (bb) {$0$}
;}\newcommand{\nodebc}{\node[draw,circle] (bc) {$11$}
;}\begin{tikzpicture}[auto]
\matrix[column sep=.3cm, row sep=.3cm,ampersand replacement=\&]{
         \&         \&         \&         \& \nodea  \&         \&         \&         \&         \\
 \nodeb  \&         \&         \&         \& \noded  \&         \&         \&         \& \nodebc \\
 \nodec  \&         \& \nodee  \&         \&         \&         \& \nodei  \&         \&         \\
         \& \nodef  \& \nodeg  \& \nodeh  \&         \& \nodej  \&         \& \nodeba \&         \\
         \&         \&         \&         \&         \&         \&         \& \nodebb \&         \\
};

\path[ultra thick, red] (b) edge (c)
	(e) edge (f) edge (g) edge (h)
	(ba) edge (bb)
	(i) edge (j) edge (ba)
	(d) edge (e) edge (i)
	(a) edge (b) edge (d) edge (bc);
\end{tikzpicture}}
    \hfil
    \scalebox{0.7}
{ \newcommand{\nodea}{\node[draw,circle] (a) {$12$}
;}\newcommand{\nodeb}{\node[draw,circle] (b) {$10$}
;}\newcommand{\nodec}{\node[draw,circle] (c) {$11$}
;}\newcommand{\noded}{\node[draw,circle] (d) {$5$}
;}\newcommand{\nodee}{\node[draw,circle] (e) {$6$}
;}\newcommand{\nodef}{\node[draw,circle] (f) {$7$}
;}\newcommand{\nodeg}{\node[draw,circle] (g) {$2$}
;}\newcommand{\nodeh}{\node[draw,circle] (h) {$0$}
;}\newcommand{\nodei}{\node[draw,circle] (i) {$1$}
;}\newcommand{\nodej}{\node[draw,circle] (j) {$8$}
;}\newcommand{\nodeba}{\node[draw,circle] (ba) {$3$}
;}\newcommand{\nodebb}{\node[draw,circle] (bb) {$9$}
;}\newcommand{\nodebc}{\node[draw,circle] (bc) {$4$}
;}\begin{tikzpicture}[auto]
\matrix[column sep=.3cm, row sep=.3cm,ampersand replacement=\&]{
         \& \nodea  \&         \&         \&         \&         \&         \&         \&         \\
 \nodeb  \&         \&         \&         \&         \& \nodec  \&         \&         \&         \\
         \&         \& \noded  \& \nodee  \&         \& \nodef  \&         \& \nodej  \& \nodebb \\
         \&         \&         \&         \&         \& \nodeg  \&         \& \nodeba \& \nodebc \\
         \&         \&         \&         \& \nodeh  \&         \& \nodei  \&         \&         \\
};

\path[ultra thick, red] (g) edge (h) edge (i)
	(f) edge (g)
	(j) edge (ba)
	(bb) edge (bc)
	(c) edge (d) edge (e) edge (f) edge (j) edge (bb)
	(a) edge (b) edge (c);
\end{tikzpicture}}
\caption{The naturally labelled trees corresponding to the Kupisch
  series $[5,6,5,4,4,3,3,3,2,3,2,1]$ and to
  $[2,1,5,5,5,6,5,4,3,2,2,1]$.}
\label{fig:tree1}
\end{figure}

An \Dfn{ordered rooted tree} is an unlabeled rooted tree together with a linear order on the children of each vertex.  Let $\mathcal T_n$ be the set of ordered rooted trees on $n+1$ vertices.  We identify $\mathcal L_n$ and $\mathcal T_n$ using the following bijection.
\begin{corollary}
  The map from naturally labeled trees to ordered rooted trees which orders the children of each vertex according to their label, and then removes all labels, is a bijection.
\end{corollary}

Recall that a \Dfn{sibling} of a vertex of a rooted tree is a vertex having the same parent.
\begin{lemma}\label{lem:treedepth}
    Let $A \in \mathcal A_n$ and let $T = \tau(A) \in \mathcal T_n$.
    Fix $g\ge 0$.
    Then
    \[
      \gldim(A) \leq g
        %\max_{0\leq i < n}\dist_G(i, i+1) \leq g + 1
    \]
    if and only if there are no two siblings $i < j$ such that the subtree rooted at $i$ has depth $\lfloor\frac g 2\rfloor$ or larger, and the subtree rooted at $j$ has depth $\lceil\frac g 2\rceil$ or larger.
\end{lemma}
\begin{proof}
  As $\gldim(A) = \max_{0\leq i < n}\dist_T(i, i+1) - 1$, we have that
  $\gldim(A) \leq g$ if and only if $\dist_T(i,i+1) \leq g+1$ for all $0 \leq i < n$.

  Assume first that there exists a vertex~$i$ with $\dist_T(i, i+1) > g+1$.
  Let~$a < b$ be the two children of the vertex closest to the root on the unique path from~$i$ to $i+1$.  By~\Cref{lem:interval}, we have $\depth_T(i+1)\in\{\depth_T(i), \depth_T(i)-1\}$.  If $\depth_T(i+1) = \depth_T(i)$ we obtain that the two subtrees rooted at~$a$ and~$b$ each have depth at least $\lceil\frac g 2\rceil$.
  Otherwise, if  $\depth_T(i+1) = \depth_T(i)-1$, the subtree rooted at $b$ has depth at least $\lceil\frac g 2\rceil$ and the subtree rooted at $a$ has depth at least $\lfloor\frac g 2\rfloor$.

  Conversely, suppose that $T$ has two siblings $a < b$ such that the subtree rooted at $a$ has depth $\lfloor\frac g 2\rfloor$ or larger and the subtree rooted at $b$ has depth $\lceil\frac g 2\rceil$ or larger.  Suppose first that $g$ is even and therefore $\lfloor\frac g 2\rfloor = \lceil\frac g 2\rceil$.  Let $i$ be the largest vertex at depth $\frac g 2$ of the subtree rooted at $a$.  Then, since the subtree rooted at $b$ also has depth at least $\frac g 2$, \Cref{lem:interval} implies that the vertex $i+1$ has the same depth as vertex $i$.  In particular, $i+1$ cannot be in the subtree rooted at $a$.  Therefore, there is a sibling $c > a$ of $a$ such that the subtree rooted at $c$ contains $i+1$, which implies that $\dist_T(i, i+1) > g + 1$.  Finally, suppose that $g$ is odd.  Let $i$ be the largest vertex at depth $\lceil\frac g 2\rceil$ of the subtree rooted at $b$.  Then either $i+1$ also has depth $\lceil\frac g 2\rceil$, or it has depth $\lfloor\frac g 2\rfloor$.  In the former case it cannot be in the subtree rooted at $b$ because $i$ was chosen to be maximal.  In the latter case, it cannot be in the subtree rooted at $b$, because the subtree rooted at $a$ has depth $\lfloor\frac g 2\rfloor$.  Thus, in both cases $\dist_T(i, i+1) > g + 1$.
\end{proof}

For example, for the first tree in \Cref{fig:tree1} we have
$\max_{0\leq i < n}\dist_T(i, i+1)=\dist_T(0, 1) = 5$, whence the corresponding Nakayama algebra has global dimension $4$.  Accordingly, there are no
two siblings~$a$ and~$b$ such that the subtrees rooted at~$a$ and~$b$
both have depth at least~$2$.  For the second tree we have
$\max_{0\leq i < n}\dist_T(i, i+1)=\dist_T(2,3) = 4$.  Indeed, there are no two siblings $a < b$ such that the subtree rooted at~$a$ has depth at least~$1$ and the subtree rooted at~$b$ has depth at least~$2$.

To prove \Cref{thm:equidistributed} we provide for every $g\in\NN$ a
bijection between the set
\[
  \big\{A\in\mathcal A_n \mid \gldim(A) \leq g\big\}
\]
and the set of Dyck paths of semilength $n$ and height at most $g+1$.
We then deduce \Cref{thm:equidistributed} by an inclusion-exclusion argument.
The bijection essentially consists of a decomposition of ordered rooted trees and a
corresponding decomposition of Dyck paths.
\begin{lemma}
\label{decompositionTree}
  For any $g\ge 0$, there is a bijection mapping an ordered rooted tree $T$ with
  $n+1$ vertices and $\max_{0\leq i < n}\dist_T(i, i+1) \leq g+1$ to a quadruple $(m, L, R, M)$, such that
  \begin{itemize}
  \item $m$ is a non-negative integer,
  \item $L=(L_1,\dots,L_m)$ is an $m$-tuple of ordered rooted trees
    of depth at most $\lfloor\frac g 2\rfloor$,
  \item $R=(R_1,\dots,R_m)$ is an $m$-tuple of ordered rooted trees
    of depth at most $\lceil\frac g 2\rceil$, and
  \item $M$ is an ordered rooted tree of depth at most
    $\lceil\frac g 2\rceil$ if $m=0$, and depth
    $\lceil\frac g 2\rceil$ otherwise,
  \end{itemize}
  and the total number of vertices of the trees in $L$, $R$ and $M$
  is $n+1+m$.
\end{lemma}
\begin{proof}
  Let $T\in\mathcal T_n$ be a tree with
  $\max_{0\leq i < n}\dist_T(i, i+1) \leq g+1$.  We decompose $T$ as follows:
  \begin{enumerate}
  \item let $m=0$.
  \item if the depth of $T$ is at most $\lceil\frac g 2\rceil$, let
    $M=T$ and stop.
  \item otherwise, by \Cref{lem:treedepth}, there is a unique child
    $a$ of the root of $T$ such that the subtree rooted at $a$ has
    depth at least $\lceil\frac g 2\rceil$.  Then,
    \begin{enumerate}
    \item increase~$m$ by~$1$,
    \item let $L_m$ be the subtree of $T$ obtained by deleting $a$
      and all its siblings to the right,
    \item let $R_m$ be the subtree of $T$ obtained by deleting $a$
      and all its siblings to the left, and
    \item let $T$ be the subtree of $T$ rooted at $a$.
    \end{enumerate}
  \item goto (2).
  \end{enumerate}
  \Cref{lem:treedepth} ensures that all trees in $L$ and $R$, as well as $M$ have the correct depth.
  T has $n+1$ vertices.  Each time we execute step~(3), one vertex -- the parent of $a$ -- is added to both $L_m$ and $R_m$.  Every other vertex appears in exactly one of the trees in $L$, $R$ and $M$.  Thus, the total number of vertices is $n+1+m$.  It is clear that this algorithm can be reversed to obtain $T$ from $(m,L,R,M)$.
\end{proof}

For example, applying this procedure to the first tree in
\Cref{fig:tree1} with $g=4$, we obtain $m=2$,

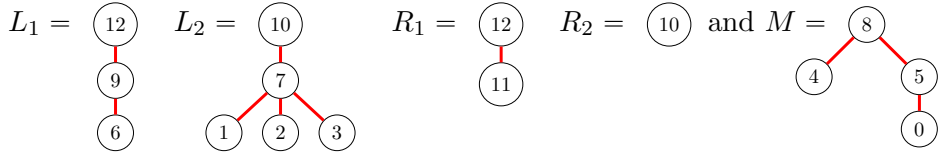
\[
L_1 = \scalebox{0.7}{ \newcommand{\nodea}{\node[draw,circle] (a) {$12$}
;}\newcommand{\nodeb}{\node[draw,circle] (b) {$9$}
;}\newcommand{\nodec}{\node[draw,circle] (c) {$6$}
;}\begin{tikzpicture}[baseline=(a.base)]
\matrix[column sep=.3cm, row sep=.3cm,ampersand replacement=\&]{
 \nodea  \\
 \nodeb  \\
 \nodec  \\
};
\path[ultra thick, red] (b) edge (c)
	(a) edge (b);
      \end{tikzpicture}}\quad
    L_2 = \scalebox{0.7}{\hspace{-1cm}
      \newcommand{\nodea}{\node[draw,circle] (a) {$10$}
;}\newcommand{\nodeb}{\node[draw,circle] (b) {$7$}
;}\newcommand{\nodec}{\node[draw,circle] (c) {$1$}
;}\newcommand{\noded}{\node[draw,circle] (d) {$2$}
;}\newcommand{\nodee}{\node[draw,circle] (e) {$3$}
;}\begin{tikzpicture}[baseline=(a.base)]
\matrix[column sep=.3cm, row sep=.3cm,ampersand replacement=\&]{
         \& \nodea  \&         \\
         \& \nodeb  \&         \\
 \nodec  \& \noded  \& \nodee  \\
};
\path[ultra thick, red] (b) edge (c) edge (d) edge (e)
	(a) edge (b);
      \end{tikzpicture}}\quad
R_1 = \scalebox{0.7}{ \newcommand{\nodea}{\node[draw,circle] (a) {$12$}
;}\newcommand{\nodeb}{\node[draw,circle] (b) {$11$}
;}\begin{tikzpicture}[baseline=(a.base)]
\matrix[column sep=.3cm, row sep=.3cm,ampersand replacement=\&]{
 \nodea  \\
 \nodeb  \\
};
\path[ultra thick, red] (a) edge (b);
\end{tikzpicture}}\quad
R_2 = \scalebox{0.7}{ \newcommand{\nodea}{\node[draw,circle] (a) {$10$}
;}\begin{tikzpicture}[baseline=(a.base)]
\matrix[column sep=.3cm, row sep=.3cm,ampersand replacement=\&]{
 \nodea  \\
};
\end{tikzpicture}}\text{ and }
M = \scalebox{0.7}{\hspace{-1cm} \newcommand{\nodea}{\node[draw,circle] (a) {$8$}
;}\newcommand{\nodeb}{\node[draw,circle] (b) {$4$}
;}\newcommand{\nodec}{\node[draw,circle] (c) {$5$}
;}\newcommand{\noded}{\node[draw,circle] (d) {$0$}
;}\begin{tikzpicture}[baseline=(a.base)]
\matrix[column sep=.3cm, row sep=.3cm,ampersand replacement=\&]{
         \& \nodea  \&         \\
 \nodeb  \&         \& \nodec  \\
         \&         \& \noded  \\
};
\path[ultra thick, red] (c) edge (d)
	(a) edge (b) edge (c);
\end{tikzpicture}}
\]

Applying the procedure to the second tree in \Cref{fig:tree1} with
$g=3$, we obtain $m=2$,
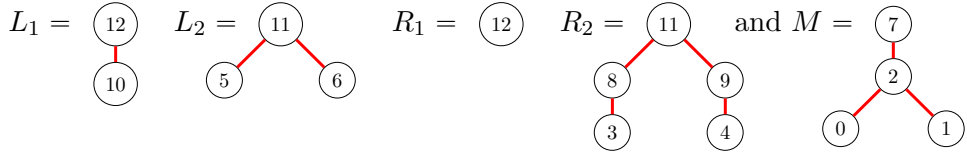
\[
  L_1 = \scalebox{0.7}{ \newcommand{\nodea}{\node[draw,circle] (a) {$12$}
;}\newcommand{\nodeb}{\node[draw,circle] (b) {$10$}
;}\begin{tikzpicture}[baseline=(a.base)]
\matrix[column sep=.3cm, row sep=.3cm,ampersand replacement=\&]{
 \nodea  \\
 \nodeb  \\
};
\path[ultra thick, red] (a) edge (b);
\end{tikzpicture}}\quad
L_2 = \scalebox{0.7}{\hspace{-1cm} \newcommand{\nodea}{\node[draw,circle] (a) {$11$}
;}\newcommand{\nodeb}{\node[draw,circle] (b) {$5$}
;}\newcommand{\nodec}{\node[draw,circle] (c) {$6$}
;}\begin{tikzpicture}[baseline=(a.base)]
\matrix[column sep=.3cm, row sep=.3cm,ampersand replacement=\&]{
         \& \nodea  \&         \\
 \nodeb  \&         \& \nodec  \\
};
\path[ultra thick, red] (a) edge (b) edge (c);
\end{tikzpicture}}\quad
R_1 = \scalebox{0.7}{ \newcommand{\nodea}{\node[draw,circle] (a) {$12$}
;}\begin{tikzpicture}[baseline=(a.base)]
\matrix[column sep=.3cm, row sep=.3cm,ampersand replacement=\&]{
 \nodea  \\
};
\end{tikzpicture}}\quad
R_2 = \scalebox{0.7}{\hspace{-1cm} \newcommand{\nodea}{\node[draw,circle] (a) {$11$}
;}\newcommand{\nodeb}{\node[draw,circle] (b) {$8$}
;}\newcommand{\nodec}{\node[draw,circle] (c) {$3$}
;}\newcommand{\noded}{\node[draw,circle] (d) {$9$}
;}\newcommand{\nodee}{\node[draw,circle] (e) {$4$}
;}\begin{tikzpicture}[baseline=(a.base)]
\matrix[column sep=.3cm, row sep=.3cm,ampersand replacement=\&]{
         \& \nodea  \&         \\
 \nodeb  \&         \& \noded  \\
 \nodec  \&         \& \nodee  \\
};
\path[ultra thick, red] (b) edge (c)
	(d) edge (e)
	(a) edge (b) edge (d);
\end{tikzpicture}}\text{\hspace{-6pt}and }
M = \scalebox{0.7}{\hspace{-1cm} \newcommand{\nodea}{\node[draw,circle] (a) {$7$}
;}\newcommand{\nodeb}{\node[draw,circle] (b) {$2$}
;}\newcommand{\nodec}{\node[draw,circle] (c) {$0$}
;}\newcommand{\noded}{\node[draw,circle] (d) {$1$}
;}\begin{tikzpicture}[baseline=(a.base)]
\matrix[column sep=.3cm, row sep=.3cm,ampersand replacement=\&]{
         \& \nodea  \&         \\
         \& \nodeb  \&         \\
 \nodec  \&         \& \noded  \\
};
\path[ultra thick, red] (b) edge (c) edge (d)
	(a) edge (b);
\end{tikzpicture}}
\]

\begin{lemma}
\label{decompositionDyckPath}
  There is a bijection between Dyck paths of semilength $n$ and
  height at most $g+1$ and quadruples $(m, L, R, M)$, such that
  \begin{itemize}
  \item $m$ is a non-negative integer,
  \item $L=(L_1,\dots,L_m)$ is an $m$-tuple of Dyck paths of height
    at most $\lfloor\frac g 2\rfloor$,
  \item $R=(R_1,\dots,R_m)$ is an $m$-tuple of Dyck paths of height
    at most $\lceil\frac g 2\rceil$, and
  \item $M$ is a Dyck path of height at most $\lceil\frac g 2\rceil$
    if $m=0$ and height $\lceil\frac g 2\rceil$ otherwise,
  \end{itemize}
  and the sum of the semilengths of the Dyck paths in $L$, $R$ and
  $M$ is $n-m$.
\end{lemma}
\begin{proof}
  Let $D$ be a Dyck path of semilength $n$ and height at most $g+1$.
  If $D$ has height $\lceil\frac g 2\rceil$ or less, $m=0$.
  Otherwise,
  \begin{enumerate}
  \item split $D$, regarded as a sequence of up and down steps, into
    $2m+1$ (possibly empty) subsequences
    $A, L_1, \bar R_2,L_2,\bar R_3, L_3, \dots, \bar R_{m}, L_m, B$, by removing
    all steps from height $\lceil\frac g 2\rceil$ to height
    $\lceil\frac g 2\rceil + 1$ and from height
    $\lceil\frac g 2\rceil + 1$ to height $\lceil\frac g 2\rceil$.
    Note that $L_k$ is a Dyck path of height at most
    $\lfloor\frac g 2\rfloor$ for $1\leq k\leq m$.
  \item for $2\leq k\leq m$, let $R_k$ be the Dyck path of height at
    most $\lceil\frac g 2\rceil$ given by reversing $\bar R_k$.
  \item let $P$ be the prefix of $A$ consisting of all steps until
    the first peak of height $\lceil\frac g 2\rceil$ is reached, and
    let $R_1$ be the reversal of what remains of $A$. Let $M$ be the
    Dyck path obtained by concatenating $P$ with the reversal of $B$.
  \end{enumerate}
  $D$ consists of $n$ up and $n$ down steps. In step (1), we remove $m$ up steps and $m$ down steps, all other steps are distributed between the Dyck paths $L$, $R$ and $M.$ Hence, the sum of the semilengths of these Dyck paths is indeed $n-m.$

  \end{proof}

Applying the procedure to the Dyck path
\[
  \begin{tikzpicture}[scale=\textwidth/25cm]
    \draw[dotted] (0, 0) grid (24, 5);
    \draw[color=red, thick, dashed] (0, 2.5) -- (24, 2.5);
    \draw[rounded corners=1, color=black, line width=2] (0, 0) -- (1, 1) -- (2, 0) -- (3, 1) -- (4, 2) -- (5, 1) -- (6, 2) -- (7, 3) -- (8, 4) -- (9, 5) -- (10, 4) -- (11, 3) -- (12, 2) -- (13, 3) -- (14, 4) -- (15, 5) -- (16, 4) -- (17, 5) -- (18, 4) -- (19, 5) -- (20, 4) -- (21, 3) -- (22, 2) -- (23, 1) -- (24, 0);
    \draw [thick, decoration={brace, mirror, raise=6pt}, decorate] (0.1, 0) -- (3.9, 0)
    node [pos=0.5,anchor=south,yshift=-30pt] {$P$};
    \draw [thick, decoration={brace, mirror, raise=6pt}, decorate] (4.1, 0) -- (5.9, 0)
    node [pos=0.5,anchor=south,yshift=-30pt] {$\bar R_1$};
    \draw [thick, decoration={brace, mirror, raise=6pt}, decorate] (7.1, 0) -- (10.9, 0)
    node [pos=0.5,anchor=south,yshift=-30pt] {$L_1$};
    \draw [thick, decoration={brace, mirror, raise=6pt}, decorate] (12, 0) -- (12, 0)
    node [pos=0.5,anchor=south,yshift=-30pt] {$\bar R_2$};
    \draw [thick, decoration={brace, mirror, raise=6pt}, decorate] (13.1, 0) -- (20.9, 0)
    node [pos=0.5,anchor=south,yshift=-30pt] {$L_2$};
    \draw [thick, decoration={brace, mirror, raise=6pt}, decorate] (22.1, 0) -- (23.9, 0)
    node [pos=0.5,anchor=south,yshift=-30pt] {$B$};
  \end{tikzpicture}
\]
and $g=4$, we obtain
\begin{align*}
  L_1 &= \begin{tikzpicture}[scale=\textwidth/25cm]
    \draw[dotted] (0, 0) grid (4, 2);
    \draw[rounded corners=1, color=black, line width=2] (0, 0) -- (1, 1) -- (2, 2) -- (3, 1) -- (4, 0);
  \end{tikzpicture},%
  &
    L_2 &= \begin{tikzpicture}[scale=\textwidth/25cm]
      \draw[dotted] (0, 0) grid (8, 2);
      \draw[rounded corners=1, color=black, line width=2] (0, 0) -- (1, 1) -- (2, 2) -- (3, 1) -- (4, 2) -- (5, 1) -- (6, 2) -- (7, 1) -- (8, 0);
    \end{tikzpicture},\\
  R_1 &= \begin{tikzpicture}[scale=\textwidth/25cm]
    \draw[dotted] (0, 0) grid (2, 1);
    \draw[rounded corners=1, color=black, line width=2] (0, 0) -- (1, 1) -- (2, 0);
  \end{tikzpicture},%
  &
    R_2 &= \emptyset,\\
  M &=
      \begin{tikzpicture}[scale=\textwidth/25cm]
        \draw[dotted] (0, 0) grid (6, 2);
        \draw[rounded corners=1, color=black, line width=2] (0, 0) -- (1, 1) -- (2, 0) -- (3, 1) -- (4, 2) -- (5, 1) -- (6, 0);
      \end{tikzpicture}
\end{align*}

Applying the procedure to the Dyck path
\[
  \begin{tikzpicture}[scale=\textwidth/25cm]
    \draw[dotted] (0, 0) grid (24, 4);
    \draw[color=red, thick, dashed] (0, 2.5) -- (24, 2.5);
    \draw[rounded corners=1, color=black, line width=2] (0, 0) -- (1, 1) -- (2, 2) -- (3, 3) -- (4, 4) -- (5, 3) -- (6, 2) -- (7, 1) -- (8, 0) -- (9, 1) -- (10, 2) -- (11, 1) -- (12, 0) -- (13, 1) -- (14, 2) -- (15, 3) -- (16, 4) -- (17, 3) -- (18, 4) -- (19, 3) -- (20, 2) -- (21, 1) -- (22, 0) -- (23, 1) -- (24, 0);
  \end{tikzpicture}
\]
and $g=3$, we obtain
\begin{align*}
  L_1 &=\begin{tikzpicture}[scale=\textwidth/25cm]
    \draw[dotted] (0, 0) grid (2, 1);
    \draw[rounded corners=1, color=black, line width=2] (0, 0) -- (1, 1) -- (2, 0);
  \end{tikzpicture},%
  &
    L_2 &= \begin{tikzpicture}[scale=\textwidth/25cm]
      \draw[dotted] (0, 0) grid (4, 1);
      \draw[rounded corners=1, color=black, line width=2] (0, 0) -- (1, 1) -- (2, 0) -- (3, 1) -- (4, 0);
    \end{tikzpicture},\\
  R_1 &= \emptyset,%
  &
    R_2 &= \begin{tikzpicture}[scale=\textwidth/25cm]
      \draw[dotted] (0, 0) grid (8, 2);
      \draw[rounded corners=1, color=black, line width=2] (0, 0) -- (1, 1) -- (2, 2) -- (3, 1) -- (4, 0) -- (5, 1) -- (6, 2) -- (7, 1) -- (8, 0);
    \end{tikzpicture},\\
  M &= \begin{tikzpicture}[scale=\textwidth/25cm]
    \draw[dotted] (0, 0) grid (6, 2);
    \draw[rounded corners=1, color=black, line width=2] (0, 0) -- (1, 1) -- (2, 2) -- (3, 1) -- (4, 2) -- (5, 1) -- (6, 0);
  \end{tikzpicture}
\end{align*}

Combining the previous lemmas, we can prove
\Cref{thm:equidistributed}, which we restate for convenience.
\begin{theorem}
  The global dimension of connected linear Nakayama algebras is
  equidistributed with the height of Dyck paths.  In symbols,
  \[
    \sum_{A} q^{\gldim(A)} = \sum_{D} q^{\height(D)},
  \]
  where the first sum ranges over all connected linear Nakayama
  algebras with~$n$ simple modules and the second sum ranges over all
  Dyck paths of semilength~$n-1$.
\end{theorem}

\begin{proof}
  There is a classical bijection between rooted ordered trees
  with~$n-1$ vertices and depth~$h$ and Dyck paths of semilength~$n$
  and height~$h$, which proceeds by traversing the tree in pre-order,
  and recording an up step whenever the edge is traversed away from
  the root and a down step otherwise.

  Together with \Cref{thm:projdim}, \Cref{decompositionTree} and
  \Cref{decompositionDyckPath}, this shows that Dyck paths of
  semilength $n$ and height at most $g+1$ are in bijection with
  (possibly disconnected) linear Nakayama algebras with global
  dimension at most $g$.

  Let
  \[
    \mathcal{D}(n,g) := \{D~\text{Dyck path of semilength
      $n$}\mid\height(D)\leq g+1\}
  \]
  and let
  \[
    \mathcal{A}(n,g) := \{A\in\mathcal A_n\mid\gldim(A)\leq g\}.
  \]

  Recall that a Dyck path is \Dfn{prime} if it only returns to the
  $x$-axis with its final step.  Let $\mathcal D^c(n,g)$ be the
  subset of prime Dyck paths in $\mathcal{D}(n,g)$.  Note that
  removing the first up and the final down step of a prime Dyck path
  provides a bijection between $\mathcal D^c(n,g)$ and
  $\mathcal D(n-1,g-1)$

  Let $\mathcal A^c(n,g)$ be the subset of connected linear Nakayama
  algebras in $\mathcal{A}(n,g)$.  We claim that
  $\lvert\mathcal A^c(n,g)\rvert = \lvert\mathcal D^c(n,g)\rvert$.

  To show this, let
  \begin{align*}
    \mathcal D(x,g) &= \sum_n \lvert\mathcal D(n,g)\rvert x^n,%
    &
      \mathcal D^c(x,g) &= \sum_n \lvert\mathcal D^c(n,g)\rvert x^n,\\
    \mathcal A(x,g) &= \sum_n \lvert\mathcal A(n,g)\rvert x^n,%
    &
      \mathcal A^c(x,g) &= \sum_n \lvert\mathcal A^c(n,g)\rvert x^n
  \end{align*}
  be the generating functions associated with the sets introduced
  above.  Since a Dyck path is a sequence of prime Dyck paths, and a
  linear Nakayama algebra is a sequence of connected linear Nakayama
  algebras, we have
  \[
    \mathcal D(x,g) = \frac{1}{1 - \mathcal D^c(x,g)}\quad\text{and}\quad
    \mathcal A(x,g) = \frac{1}{1 - \mathcal A^c(x,g)}.
  \]
  Since $\mathcal D(x,g) = \mathcal A(x,g)$, the claim follows.

  Finally, we note that the statement of the theorem is equivalent to
  the equality
  \[
    \lvert\mathcal A^c(n,g)\rvert - \lvert\mathcal A^c(n,g-1)\rvert %
    = \lvert\mathcal D(n-1,g-1)\rvert %
    - \lvert\mathcal D(n-1,g-2)\rvert.
  \]
\end{proof}

\section{Sincere Nakayama algebras}\label{sec:sincere}

We first look at the magnitude of Nakayama algebras and then focus on Nakayama algebras where every projective module is sincere. In this section the magnitude, defined as the sum of all entries of the inverse of the Cartan matrix of a finite-dimensional algebra $A$ with finite global dimension is denoted by $m_A$.
This notion was first introduced and studied in \cite{CKL}, to which we refer for more information.

\subsection{The magnitude of Nakayama algebras}

\begin{theorem}\label{CKLtheorem}
Let $A$ be a finite-dimensional algebra with Jacobson radical $J$ and finite global dimension $g$.
Then $m_A= \sum\limits_{k=0}^{g}(-1)^l \dim\Ext_A^k(A/J,A/J)$.

\end{theorem}
\begin{proof}
This is the main result of \cite{CKL}.
\end{proof}

\begin{proposition} \label{magnitudepropo}
Let $A$ be a Nakayama algebra of finite global dimension.
Then $m_A$ is equal to the number of simple modules of even projective dimension, and also to the number of vertices lying on a cycle in the resolution quiver of $A$.
\end{proposition}

\begin{proof}
Assume that $A$ has finite global dimension $g$.  We are going to apply \Cref{CKLtheorem} and therefore compute $\dim\Ext_A^k(A/J,A/J)$.

Since every indecomposable module over a Nakayama algebra is uniserial, any term in the minimal projective resolution of an indecomposable module is zero or indecomposable.
Let $S_i$ and $S_j$ be two simple $A$-modules, and let
\[
\dots\rightarrow P_k\rightarrow\dots\rightarrow P_1 \rightarrow P_0 \rightarrow S_i \rightarrow 0
\]
be the minimal projective resolution of $S_i$.  Since $S_i$ is simple, \Cref{benpropo} implies that the vector space dimension of $\Ext_A^k(S_i,S_j)$ is one if $P_k$ is isomorphic to the projective cover of $S_j$ and zero otherwise.  Since $S_j$ is simple, $P_k$ is isomorphic to the projective cover of $S_j$ if and only if the top of $P_k$ is isomorphic to $S_j$.

Since $A/J$ is the direct sum of all simple $A$-modules, we have
\[
  \Ext_A^k(S_i, A/J) = \bigoplus_j \Ext_A^k(S_i, S_j),
\]
which is one dimensional if $\pdim(S_i)\geq k$ and zero dimensional otherwise.  It follows that $\dim \Ext_A^k(A/J, A/J)= \lvert\{\text{$S$ simple} \mid \pdim(S) \geq k \}\rvert$.

By \Cref{CKLtheorem},
\begin{align*}
    m_A &= \sum_{k=0}^{g}(-1)^l \dim\Ext_A^k(A/J,A/J) \\
    &= \sum_{\text{$k$ even}}\lvert\{\text{$S$ simple} \mid \pdim(S) \geq k \}\rvert - \sum_{\text{$k$ odd}}\lvert\{\text{$S$ simple} \mid \pdim(S) \geq k \}\rvert\\
    &= \sum_{\text{$k$ even}}\lvert\{\text{$S$ simple} \mid \pdim(S) = k \}\rvert
\end{align*}
Thus, $m_A$ is equal to the number of simple modules of even projective dimension.  By \Cref{Dshentheorems}~(3), this is also the number of vertices lying on a cycle in the resolution quiver of $A$.
\end{proof}

\begin{remark}
\label{rmk::Madsen}

We remark that in \cite{Mad}, Madsen proved that a Nakayama algebra has finite global dimension if and only if there is a simple module of even projective dimension. In \cite{SW} a systematic study of the magnitude for Nakayama algebras can be found, which also includes an alternative proof for the previous proposition.
\end{remark}
We call a Nakayama algebra $A$ an \Dfn{M1-Nakayama algebra} if $m_A=1$.
We now give a classification of M1-Nakayama algebras using the previous proposition:
\begin{proposition} \label{magnitudetheorem}
The following are equivalent for a Nakayama algebra $A$ with $n$ simple modules.
\begin{enumerate}
\item $A$ has finite global dimension and magnitude one.
\item $A$ has a unique indecomposable projective module of dimension $n$.
\item $A$ has finite global dimension and Loewy length at least $n$.
\end{enumerate}
\end{proposition}

\begin{proof}
The equivalences are trivial for linear Nakayama algebras.  In this case the unique algebra satisfying any of the properties has Kupisch series $[n,n-1,\dots,3,2,1].$  Thus, we assume that $A$ is a cyclic Nakayama algebra.

Let us first show that $A$ has finite global dimension if it has a unique indecomposable projective module of dimension $n$.  This module corresponds to a loop of weight~$1$ in the resolution quiver.  By \Cref{Dshentheorems}~(1), any further cycle in the resolution quiver would then also be a loop of weight $1$, and therefore also correspond to indecomposable projective modules of dimension $n$.  Thus, the resolution quiver is in fact connected and its unique cycle has weight $1$.  \Cref{Dshentheorems}~(2) now implies that $A$ has finite global dimension.

Moreover, in this case \Cref{magnitudepropo} implies that $A$ has magnitude one, and, trivially, the Loewy length of $A$ is at least $n$.

Assume now that $A$ has finite global dimension and magnitude one.  By \Cref{magnitudepropo}, the magnitude is equal to the number of vertices in the resolution quiver of $A$ that lie on a cycle, so the unique cycle in the resolution quiver must be a loop.  Thus, there is a unique $i$ with $c_i=\dim e_i A\equiv 0 \mod n.$  By \Cref{gustafsonbound}, we have that $c_j \leq 2n-1$ for all $j$, so $c_i=n$.

Finally, assume that $A$ has finite global dimension and the Loewy length of $A$ is at least $n$.  By \cite[Lemma 3.2. (1)]{MMZ}, a Nakayama algebra of finite global dimension must have at least one $c_i$ with $c_i\le n$.  Because of \Cref{cnak_eq} every value between the minimal entry and the maximal entry of the Kupisch series of $A$ must be attained.  Therefore, there is an $i$ with $c_i=n$.  Since the resolution quiver of $A$ is connected and has weight $1$ by \Cref{Dshentheorems}~(2), there are no further entries equal to~$n$ in the Kupisch series of $A$.
\end{proof}

\begin{corollary}
The number of Nakayama algebras with $n$ simple modules of finite global dimension and magnitude one is given by the Catalan number $C_n$.
\end{corollary}
\begin{proof}
We provide a bijection between Kupisch series with exactly one entry equal to $n$ and Dyck paths of semilength $n$.

Let $[c_0,\ldots,c_{n-1}]$ be the Kupisch series of a Nakayama algebra $A$ and suppose it contains exactly one entry equal to $n$.  Without loss of generality, we can assume that $c_0=n$.  This is the case if $A$ is a linear Nakayama algebra.  Otherwise, every cyclic rotation of the Kupisch series corresponds to a Nakayama algebra isomorphic to $A$.

Suppose that $c_i < n$ for some $i \geq 1$.  Then, because $c_{i+1} \geq c_i - 1$ and $c_{i+1} \neq n$, also $c_{i+1} > n$.  Thus, if $k$ is maximal such that $c_k\leq n$, we have $c_i \leq n$ for $0\leq i\leq k$, and $c_i > n$ for $k < i \leq n-1$.

Let
\[
c^\prime_i =
\begin{cases}
    n + 2 - c_{k-i} & \text{for $0\leq i\leq k$}\\
    c_i - n + 1 & \text{for $k < i\leq n-1$}\\
    1 & \text{for $i=n$}.
\end{cases}
\]
We claim that $[c^\prime_0,\dots,c^\prime_n]$ is the area sequence of a Dyck path.  It is immediate that $c^\prime_i\geq 2$ for $i<n$, so it remains to show that $c^\prime_{i+1} + 1 \geq c^\prime_i$ for $0\leq i < n$.

For $0\leq i< k$ we have $c^\prime_{i+1} + 1 = n+3-c_{k-i-1}\geq n+2-c_{k-i} = c^\prime_i$.  If $k=n-1$ we have $c^\prime_{k+1}+1=2=n+2-c_0=c^\prime_k$.  Otherwise, if $k<n-1$, we have \begin{align*}
    c^\prime_{k+1} + 1 &= c_{k+1} - n + 2\geq 2 = n + 2 - c_0 = c^\prime_k\\
    c^\prime_{i+1} + 1 &= c_{i+1} - n + 2 \geq c_i - n + 1 = c^\prime_i \text{ for $k < i < n-1$, and}\\
    c^\prime_n + 1 &= 2 \geq c_{n-1} - n + 1 = c^\prime_{n-1}.
\end{align*}

Therefore, the map sending the Kupisch sequence to the Dyck path with area sequence $[c^\prime_0,\dots,c^\prime_n]$ is well-defined.  To show that it is invertible, note that $c^\prime_i > 2$ for $i < k$ and that $c^\prime_k = 2$.  In fact, given any area sequence $[c^\prime_0,\dots,c^\prime_n]$, let $k$ be the minimal index such that $c^\prime_k = 2$.  Then, the Kupisch series
\[
c_i = \begin{cases}
    n + 2 - c^\prime_{k-i} & \text{for $0\leq i\leq k$}\\
    c^\prime_i - n + 1 & \text{for $k < i\leq n-1$},
\end{cases}
\]
is the preimage of the given area sequence.
\end{proof}

\subsection{The global dimension of sincere Nakayama algebras}

Finally, we focus on sincere Nakayama algebras. We call a finite-dimensional algebra $A$ is \Dfn{sincere} if every indecomposable projective $A$-module $P$ is sincere, meaning that every simple module occurs at least once in the composition series of $P$.
Sincere modules play an important role in the representation theory of Artin algebras, see for example \cite[Chapter IX.]{ARS}.

Our final theorem gives a classification of sincere Nakayama algebras of finite global dimension by showing that they are in a bijection with Dyck paths. Furthermore, it provides an explicit formula for their global dimensions using the bounce count of the corresponding Dyck paths. Note that every sincere Nakayama algebra is cyclic.

For the next theorem, we need the following proposition of Ringel, see \cite[3.3]{Rin2}.

\begin{proposition} \label{Ringelproposition}
Let $A$ be a Nakayama algebra of finite global dimension and let $M$ be an indecomposable $A$-module.
Then $M$ has odd projective dimension if and only if all composition factors of $M$ have odd projective dimension.
\end{proposition}

\begin{theorem}
\label{gldim_formula}
Let $A$ be a sincere Nakayama algebra with $n$ simple modules.  Then $A$ has finite global dimension if and only if $A$ has a unique indecomposable projective module of dimension $n$.

In this case, the map sending the Kupisch series $[c_0,\ldots,c_{n-1}]$, with $c_0=n$, of a sincere Nakayama algebra of finite global dimension $A$ to the Dyck path $D$ with area sequence
\[
  [c_1-n+1, c_2-n+1, \ldots, c_{n-1}-n+1, 1]
\]
is a bijection.  Moreover,
\[
\gldim A = 2\cdot b_D,
\]
where $b_D$ is the bounce count of the Dyck path~$D$.
\end{theorem}
\begin{proof}
  The first statement is a consequence of \Cref{magnitudetheorem},
  since a sincere Nakayama algebra has $c_i \geq n$ for all $i$, and
  thus has Loewy length at least $n$.

  It is an immediate consequence of the defining
  property~\eqref{cnak_eq} of Kupisch series of cyclic Nakayama
  algebras, and the fact that $c_i > n$ for $i\nequiv 0 \mod n$ and
  $c_n=c_0=n$ that
  $[c'_0,c'_1,\dots,c'_{n-1}] = [c_1-n+1, c_2-n+1, \ldots,
  c_{n-1}-n+1, 1]$ is the area sequence of a Dyck path $D$ of
  semilength~$n-1$, that is,
  \begin{gather*}
    c'_{i+1}+1\ge c'_i \ge 2 \text{ for } 0\le i\le n-1,\text{ and}\\
    c'_{n-1}=1.
  \end{gather*}
  Conversely, adding $n-1$ to each element of the area sequence of a
  Dyck path we obtain a Kupisch sequence $[c_0,\dots, c_{n-1}]$ that
  additionally satisfies $c_i\geq n$ for all $i$ and thus corresponds
  to a sincere Nakayama algebra.

  It remains to determine the global dimension of $A$.  We will in
  fact compute the projective dimension of $S_0$ and then show that
  this is the global dimension of $A$.  Let
  \[
    \ldots \longrightarrow P_2 \longrightarrow P_1 \longrightarrow P_0 \longrightarrow S_0 \longrightarrow 0
  \]
  be a minimal projective resolution of $S_0$ and let
  $0 = b_0 < b_1< b_2 <\ldots < b_d = n - 1$ be the bounce points of
  the bounce path associated with the Dyck path $D$, defined
  inductively via $b_{t+1} = b_t + c'_{b_t} - 1$.  We will show that
  \begin{equation}
    \label{resolution}
    P_m\cong
    \begin{cases}
      e_0 A &\text{for $0\leq m=2t\leq 2d$,}\\
      e_{b_t+1} A &\text{for $1\leq m=2t+1< 2d$,}\\
      0 &\text{for $m>2d$.}\\
    \end{cases}
  \end{equation}
  In particular, $\pdim(S_0)=2d$.

  We prove \cref{resolution} by induction on $m$.  Recall that
  $P_m \cong e_i A$ if and only if the $m$-th syzygy $\Omega^m(S_0)$
  is of the form $e_i A/e_i J^k$ for some $1\leq k\leq c_i$.  In
  particular, $P_0\cong e_0 A$.  Using the formula
  \begin{equation}
    \label{omega}
    \Omega(e_iA/e_iJ^k)\cong e_{i+k}A/e_{i+k}J^{c_i-k}
  \end{equation}
  we obtain that
  $\Omega(S_0) = \Omega(e_0 A/e_0 J^1) = e_1 A/e_1 J^{c_0-1}$ and
  therefore $P_1 = e_1 A$.  Note that $\Omega(S_0)$ is not
  projective, because $c_0-1 = n-1 < c_1$.

  Suppose now that $m=2t$ is even, $P_m \cong e_0 A$ and
  $P_{m+1} \cong e_{b_t+1} A$, and assume that $\Omega^m(S_0)$ and
  $\Omega^{m+1}(S_0)$ are not projective.  Using \cref{omega} we
  then obtain, for some $1\leq k < n$, that
  \begin{align*}
    \Omega^m(S_0) & \cong e_0 A / e_0 J^k,%
    &&\text{because $P_m \cong e_0 A$,}\\
    \Omega^{m+1}(S_0) & \cong e_k A / e_k J^{n - k},%
    &&\text{because $c_0 = n$, and}\\
    \Omega^{m+2}(S_0) & \cong e_0 A / e_0 J^{c_k - n + k},%
    &&\text{because $e_n = e_0$, and finally}\\
    \Omega^{m+3}(S_0) & \cong e_{c_k - n + k} A / e_{c_k - n + k} J^{2n-c_k-k}.
  \end{align*}
  In particular, $P_{m+2}\cong e_0 A$.

  Let us first show that $\Omega^{m+2}(S_0)$ is projective if and
  only if $b_{t+1}=n-1$, that is, if $t+1=d$.  Indeed, since
  $P_{m+1} \cong e_{b_t+1} A$ we infer $k=b_t+1$.  This implies
  that
  \begin{equation}\label{eq:bounce}
    c_k - n + k = c_{b_t+1} - n + b_t + 1 = c'_{b_t} + b_t = b_{t+1} + 1.
  \end{equation}
  Thus, $e_0 A / e_0 J^{c_k - n + k} = e_0 A / e_0 J^{c_0}$ if and
  only if $c_0 = n = b_{t+1} + 1$.

  If $\Omega^{m+2}(S_0)$ is not projective, \cref{eq:bounce}
  additionally implies that $P_{m+3}\cong e_{b_{t+1}+1} A$, as
  claimed.  Furthermore, in this case
  \[
    \Omega^{m+3}(S_0)\cong e_{b_{t+1}+1} A / e_{b_{t+1}+1}
    J^{n-1-b_{t+1}}
  \]
  cannot be projective either, since
  $0 < n-1-b_{t+1} < n \leq c_{b_{t+1} + 1}$.  This establishes
  \cref{resolution}.

  Finally, we show that $\gldim A = \pdim S_0$.  Since $A$ has finite
  global dimension, there is a unique component in the resolution
  quiver of $A$, and thus also a unique cycle.  Since $c_0=n$, this
  unique cycle is a loop.  Since the simple modules in a cycle are
  exactly those of even projective dimension, $S_0$ is the unique
  simple module of even projective dimension.  Now we use
  \Cref{Ringelproposition} to show that every indecomposable
  injective $A$-module has even projective dimension. To see this,
  assume $I$ is an indecomposable injective $A$-module. Since $A$ is
  sincere, so is $A^{op}$, and therefore every simple $A$-module
  appears as a composition factor of $I$.  Thus $I$ is sincere.
  Since $S_0$ has even projective dimension, $I$ cannot have odd
  projective dimension by \Cref{Ringelproposition}.  Recall that
  \[
    \gldim A= \pdim D(A)= \max \{ \pdim I \mid I \
    \text{indecomposable injective}
    \}
  \]
  for any finite-dimensional algebra $A$ of finite global dimension,
  see for example \cite[Chapter VI.5, Lemma 5.5]{ARS}.  Thus, the
  global dimension is even, and by Auslander's \Cref{Auslanderthm} we
  must have $\gldim A = \pdim S_0$, since $S_0$ is the unique simple
  module of even projective dimension.
\end{proof}

\begin{example}
  Consider the sincere Nakayama algebra $A$ with Kupisch series
  \[
    [c_0,\ldots,c_{11}] = [12, 14, 16, 16, 16, 15, 14, 15, 14, 14, 14, 13].
  \]
  It has $12$ simple modules and exactly one indecomposable
  projective module of dimension $12$, namely $P(0)$.  Thus, we can
  calculate its global dimension using \Cref{gldim_formula}. For this, we have to look at the Dyck path $D$
  with area sequence $[3,5,5,5,4,3,4,3,3,3,2,1]$ and its bounce
  path. These are visualised in \Cref{exampleDyckBounce}. From there,
  we can determine that the bounce count is $b_D=4,$ and thus,
  $\gldim(A)=2\cdot4=8.$
\end{example}
\begin{example}
  Consider the Dyck path~$D$ with area sequence $[3,4,4,3,2,1]$.  As
  demonstrated in \Cref{example2}, $b_D=2$. \Cref{gldim_formula}
  shows that $D$ corresponds to a sincere Nakayama algebra $A_D$ with
  a unique projective module of dimension $6$.  Thus, $A_D$ must have
  Kupisch series $[6,8,9,9,8,7]$.  Computing the minimal projective
  resolutions of the simple modules, we find that $S(0)$ has the
  longest resolution:
  \[
    0\rightarrow P(0) \rightarrow P(3) \rightarrow P(0) \rightarrow
    P(1) \rightarrow P(0) \rightarrow S(0)\rightarrow 0.
  \]
  Thus, the global dimension of $A_D$ is indeed $2b_D = 4$.
\end{example}

\end{document}